\documentclass[preprint,12pt]{elsarticle}

 \usepackage{graphics}
 \usepackage{graphicx}
 \usepackage{mathptmx}      % use Times fonts if available on your TeX system
\usepackage{latexsym}
\usepackage{amssymb}
 \usepackage{amsthm}
 \usepackage{color}
 \usepackage{subfig}
\journal{Journal of Computational and Applied Mathematics}

\begin{document}
\newtheorem{thm}{Theorem}
\newtheorem{lem}[thm]{Lemma}
\newtheorem{cor}[thm]{Corollary}
\newdefinition{rmk}{Remark}
\newproof{pf}{Proof}
\newproof{pot}{Proof of Theorem \ref{thm2}}

\begin{frontmatter}



\title{A first passage problem for a bivariate diffusion process: numerical
solution with an application to neuroscience.}


\author[ben]{Elisa Benedetto\corref{cor1}}
\ead{elisa.benedetto@unito.it}
\address[ben]{Department of Mathematics \lq\lq G. Peano\rq\rq, U\-ni\-ver\-si\-ty of To\-ri\-no, Via Carlo Al\-ber\-to 10, Turin, Italy}

\author[ben]{Laura Sacerdote}
\ead{laura.sacerdote@unito.it}
\author[ben]{Cristina Zucca\fnref{fn1}}
\ead{cristina.zucca@unito.it}
\fntext[fn1]{Corresponding author: Cristina Zucca, Department of Mathematics \lq\lq G. Peano\rq\rq, U\-ni\-ver\-si\-ty of To\-ri\-no, Via Carlo Al\-ber\-to 10, Turin, Italy, tel: +39-011-6702919, fax: +39-011-6702878.}

\begin{abstract}
We consider a bivariate diffusion process and we study the first passage time of one component through a boundary. We prove that its probability density is the unique solution of a new integral equation and we propose a numerical algorithm for its solution. Convergence properties of this algorithm are discussed and the method is applied to the study of the integrated Brownian Motion and to the integrated Ornstein Uhlenbeck process. Finally a model of neuroscience interest is also discussed.  
\end{abstract}


\end{frontmatter}

% \linenumbers

\section{Introduction} \label{intro}
First passage time problems arise in a variety of applications ranging from
finance to biology, physics or psychology (\cite{smith,redner,RiccLectureNotes} and
examples cited therein). They have been largely studied (see \cite{SacGir} for a review on the subject):
analytical \cite{giornoSimmetry,ricc,NobileRiccSacExponential,Peskir,RicSacSato,Sac}, numerical or approximate results 
\cite{BuonocoreEtAl,dinardoRic,Borok,Durbin,Durbin2,RicciardiReview,TorresEtAl,wangPoltzberger,sacTomassetti} exist for specific classes of processes such as one dimensional
diffusions or Gaussian processes. On the contrary, the case of bivariate processes has not been widely studied yet. Indeed, results are available only for specific problems such as the first exit time of the considered two-dimensional
process from a specific surface \cite{giornoSymmetrytwoDim,lachal}. However
there is a set of instances where the random variable of interest is the
first passage time of one of the components of the bivariate process through a
constant or a time dependent boundary. Examples of this type of problems are the First Passage Time (FPT) of integrated processes such as the Integrated
Brownian Motion (IBM) or the Integrated Ornstein Uhlenbeck Process (IOU). Indeed, these one dimensional processes
should be studied as bivariate processes if the Markov property has to be preserved. Recent examples of applications of the IBM or of the IOU
processes have appeared in the metrological literature \cite{panfiloTavellaZucca} where these processes are alternatively used to
model the error of atomic clocks. In that case the crossing problem corresponds
to the first attainment of an assigned value by the atomic clock error.
Another application for this type of problems arises in neuroscience for the study of two-compartment models \cite{Lansky}.
Indeed, the membrane potential evolution of two communicating parts
of the neuron, the dendritic zone and the soma, can be depicted by a
two-dimensional diffusion process, whose components describe the two
considered zones. Furthermore, the time of a spike, i.e. the time when the membrane
potential changes its dynamics with a sudden hyperpolarization, is described
as the FPT of the second component through a boundary.
Motivated by these applications, we consider the FPT of one
component of a bivariate diffusion process through an assigned constant boundary, we prove an
integral equation for this distribution and we propose a numerical algorithm
for its solution. In Section 2 we introduce the notations and the necessary
mathematical background. In Section 3 we present the new integral equation and
the condition for the existence and uniqueness of its solution. In 
Section 4 we introduce a numerical algorithm for its solution and show its convergence properties.
In Section 5 we illustrate the proposed
numerical method through a set of examples, including the two-compartment model of a neuron. 
Finally in Section 6 we compare computational effort and reliability of the proposed numerical method with a totally simulation algorithm.

\section{Notations and Mathematical Background}\label{mathBack}

Let $\mathbf{X}(t)=\left(X_1(t),X_2(t)\right)'$, $t\geq t_0$, be a two-dimensional diffusion process on $I\subseteq \mathbb{R}^2$ originated in $\mathbf{X}(t_0)=\mathbf{y}$, where the superscript $'$ denotes the transpose of the vector.
Let $s<t$, we denote with
\begin{equation}\label{pdf}
f\left(\mathbf{x},t\left|\right.\mathbf{y},s\right)=
\frac{\partial^2}{\partial x_1\partial x_2}\mathbb{P}\left(\mathbf{X}(t)\leq \mathbf{x}\left|\mathbf{X}(s)=\mathbf{y}\right.\right)
\end{equation} 
its transition probability density function, where $\mathbf{x}=(x_1,x_2)$ and $\mathbf{y}=(y_1,y_2)$.\\

\noindent
In this paper we are concerned with the random variable FPT of the first component of the process $\mathbf{X}(t)$ through a boundary $S>y_1$:
$$T=\inf \left\{t\geq t_0 : X_1(t) \geq S\right\} .$$
To describe $T$ we will use its probability density function
\begin{equation} \label{g}
g\left(t\left|\mathbf{y},t_0\right.\right)=\frac{\partial}{\partial t}\mathbb{P}\left(T < t\left|\mathbf{X}(t_0)=\mathbf{y}\right.\right).
\end{equation}
Furthermore let $Z(t)$ be a random variable whose distribution coincides with the conditional distribution of $X_2(t)$ given $T=t$
\begin{equation} \label{densCond}
\mathbb{P}\left(X_2(T)<z\left|T=t;\mathbf{X}(t_0)=\mathbf{y}\right.\right).
\end{equation}

\noindent
We will also consider the joint distribution of $\left(X_2(T) , T\right)$ and its probability density function
\begin{equation} \label{gc}
g_c\left((S,z),t\left|\mathbf{y},t_0\right.\right)=
\frac{\partial^2}{\partial z\partial t}\mathbb{P}\left(X_2(T)<z, T<t\left|\mathbf{X}(t_0)=\mathbf{y}\right.\right) \mbox{, }z\in\mathbb{R}\mbox{, }t\in[t_0,\infty].
\end{equation}
We denote with $\mathbb{E}_X(h(X))$ the expectation with respect to the probability measure induced by the random variable $X$. We skip the subscript if there is no possibility of misunderstanding.
\noindent
In the following theorem we link (\ref{pdf}) with (\ref{gc}). 
\begin {thm}
For $x_1>S$ it holds
\begin{eqnarray}\label{FortetInt}
&&\mathbb{P}\left(\mathbf{X}(t)>\mathbf{x}|\mathbf{X}(t_0)=\mathbf{y}\right)\\
&&= \int_{t_0}^t d\vartheta \int_{-\infty}^{+\infty} g_c\left((S,z),\vartheta\left|\right.\mathbf{y},t_0\right) 
\mathbb{P}\left(\mathbf{X}(t)>\mathbf{x}|X_1(\vartheta)=S,X_2(\vartheta)=z\right)dz.\nonumber
\end{eqnarray}
If the joint probability density function $f\left(\mathbf{x},t\left|\right.\mathbf{y},t_0\right)$ exists, it holds
\begin{equation}\label{Fortet}
f\left(\mathbf{x},t\left|\right.\mathbf{y},t_0\right) = \int_{t_0}^t d\vartheta \int_{-\infty}^{+\infty} g_c\left((S,z),\vartheta\left|\right.\mathbf{y},t_0\right) f\left(\mathbf{x},t\left|\right.(S,z),\vartheta\right) dz .
\end{equation}
\end{thm}
\begin{proof}
Equation (\ref{FortetInt}) is a consequence of the strong Markov property, as explained in the following.\\
Let $h:(S,\infty)\times \mathbb{R}\rightarrow\mathbb{R}$ be a bounded, Borel measurable function and let $\mathcal{F}_T$ be the $\sigma$-algebra generated by the process $\mathbf{X}(t)$  up to the random time $T$. We get
\begin{eqnarray}
\mathbb{E}[h(\mathbf{X}(t))|\mathbf{X}(t_0)=\mathbf{y}]%&=&\mathbb{E}_{(\mathbf{y},t_0)}[\mathbb{E}_{(\mathbf{y},t_0)}[h(\mathbf{X}(t))|T]]\\
&=&\mathbb{E}[\mathbb{E}[h(\mathbf{X}(t))| \mathcal{F}_T;\mathbf{X}(t_0)=\mathbf{y}]]\\
&=&\mathbb{E}[\mathbb{E}[h(\mathbf{X}(t))|\mathbf{X}(T)]]\nonumber\\
&=&\int_{t_0}^td\vartheta\int_{-\infty}^{+\infty}\mathbb{E}[h(\mathbf{X}(t))|\mathbf{X}(\vartheta)=(S,z)]g_c((S,z),\vartheta|\mathbf{y},t_0) dz \nonumber
\end{eqnarray}
where the first equality uses the double expectation theorem while the second one uses the strong Markov property.
Choosing $h(\mathbf{y})=I_{\{x_1,\infty\}\times\{x_2,\infty\}}(\mathbf{y})$ we get (\ref{FortetInt}). Finally, writing the conditional probability $\mathbb{P}\left(\mathbf{X}(t)>\mathbf{x}|X_1(\vartheta)=S,X_2(\vartheta)=z\right)$ as a double integral, changing the order of integration and differentiating (\ref{FortetInt}) with respect to $x_1$ and $x_2$ we get (\ref{Fortet}).
\end{proof}

\begin{rmk}
Equation (\ref{Fortet}) was introduced without proof in \cite{giornoSymmetrytwoDim}.
\end{rmk}

The transition probability density function (\ref{pdf}) is known in a few instances. One of these cases is a process solution of \textit{linear (in the narrow sense) stochastic differential equation} \cite{Arnold}
\begin{equation}\label{sde}
\left\{
      \begin{array}{lr}
      d\mathbf{X}(t) = \left[\mathbf{A}(t)\mathbf{X}(t)+\mathbf{M}(t)\right]dt+\mathbf{G}(t)d\mathbf{B}(t),&\hspace{5mm} t\geq t_0\\
      \hspace{3mm}\\
      \mathbf{X}(t_0) =  \mathbf{y}
      \end{array}
\right.
\end{equation}
where $\mathbf{A}(t)$ and $\mathbf{G}(t)$ are $2\times 2$ matrices, $\mathbf{M}$ is a vector of $2$ components and $\mathbf{B}(t)$ is a bivariate standard Brownian motion.\\
The solution of (\ref{sde}), corresponding to the initial value $\mathbf{y}$ at instant $t_0$, is 
\begin{equation}
\mathbf{X}(t)=\mathbf{\phi}(t,t_0)\left[\mathbf{y}+\int_{t_0}^{t}\mathbf{\phi}(u,t_0)^{-1}\mathbf{M}(u)du + \int_{t_0}^{t}\mathbf{\phi}(u,t_0)^{-1}\mathbf{G}(u)d\mathbf{B}(u)\right],
\end{equation}
where $\mathbf{\phi}(t,t_0)$ is the solution of the homogeneous matrix equation
\begin{equation}
\frac{d}{dt}\mathbf{\phi}(t,t_0)=\mathbf{A}(t)\mathbf{\phi}(t,t_0)  \hspace{1mm}\mbox{ with }\hspace{1mm} \mathbf{\phi}(t_0,t_0)=\mathbf{I} \mbox{.}
\end{equation}
For $t\in[0,\infty]$, the diffusion process has a two-dimensional normal distribution with expectation vector
\begin{equation} \label{vector}
\mathbf{m}(t\left|\mathbf{y},t_0\right.):=\mathbb{E}(\mathbf{X}(t)|\mathbf{X}(t_0)=\mathbf{y})=\mathbf{\phi}(t,t_0)\left[\mathbf{y}+\int_{t_0}^{t}\mathbf{\phi}(u,t_0)^{-1}\mathbf{M}(u)du\right]
\end{equation}
and $2\times2$ conditional covariance matrix
\begin{eqnarray} \label{matrix}
\mathbf{Q}(t\left|\mathbf{y},t_0\right.)=\mathbf{\phi}(t,t_0)\left[\int_{t_0}^{t}\mathbf{\phi}(u,t_0)^{-1}\mathbf{G}(u)\mathbf{G}(u)^{'}(\mathbf{\phi}(u,t_0)^{-1})^{'}du\right]\mathbf{\phi}(t,t_0)^{'}, 
\end{eqnarray}
where the superscript $'$ denotes the transpose of the matrix.

\noindent
In the autonomous case ($\mathbf{A}(t)=\mathbf{A}$, $\mathbf{M}(t)=\mathbf{M}$ and $\mathbf{G}(t)=\mathbf{G}$), expressions (\ref{vector}) and (\ref{matrix}) are simplified
\begin{eqnarray}\label{newvector}
\mathbf{m}(t\left|\mathbf{y},t_0\right.)&=&e^{\mathbf{A} (t-t_0)}\left[\mathbf{y}+\int_{t_0}^{t}e^{-\mathbf{A} (u-t_0)}\mathbf{M}du\right]
\end{eqnarray}
\begin{eqnarray}\label{newmatrix}
\mathbf{Q}(t\left|\mathbf{y},t_0\right.)&=&e^{\mathbf{A}(t-t_0)}\left[\int_{t_0}^{t}e^{-\mathbf{A}(u-t_0)}\mathbf{G}\mathbf{G}^{'}
                      e^{-\mathbf{A}^{'}(u-t_0)}du\right]e^{\mathbf{A}^{'}(u-t_0)}\\ \nonumber
                   &=&\int_{t_0}^{t}e^{\mathbf{A} (t-u)}\mathbf{G}\mathbf{G}^{'}e^{\mathbf{A}^{'}(t-u)}du.
\end{eqnarray}

\noindent
For Gaussian or constant initial condition, the solution of (\ref{sde}) is itself a Gaussian process, frequently known as \textit{Gauss-Markov process.}\\
Examples of Gauss-Markov processes are the Integrated Brownian Motion (IBM), the Integrated Ornstein Uhlenbeck Process (IOU). Also the underlying process of the two-compartment neural model \cite{Lansky} is a Gauss-Markov process.\\

\noindent
If $\det \mathbf{Q}(t\left|\right.\mathbf{y},t_0)\neq 0$ for each $t$, the probability density function $f\left(\mathbf{x},t\left|\right.\mathbf{y},t_0\right)$ of any two-dimensional Gauss-Markov process is
$$f\left(\mathbf{x},t\left|\right.\mathbf{y},t_0\right)=\frac{\exp\left\{-\frac{1}{2}\left[\mathbf{x}-\mathbf{m}(t\left|\mathbf{y},t_0\right.)\right]^{'}\mathbf{Q}(t\left|\mathbf{y},t_0\right.)^{-1}\left[\mathbf{x}-\mathbf{m}(t\left|\mathbf{y},t_0\right.)\right]\right\}}{2\pi\sqrt{\det \mathbf{Q}(t\left|\right.\mathbf{y},t_0)}}$$
and follows the Chapman-Kolmogorov equation \cite{RiccLectureNotes}.\\

\section{An Integral Equation for the FPT Distribution} \label{FPT distrib}

Let us consider a diffusion process $\left\{\mathbf{X}(t),t\geq 0\right\}$ originated in $\mathbf{y}=\mathbf{0}$ at $t_0=0$.
It holds
\begin{thm} \label{esistenza e unicit}
If 
\begin{equation} \label{hypothesis}
\mathbb{P}\left(X_1(t)\geq S\left|X_1(\vartheta)=S,X_2(\vartheta)=z\right.\right) \mbox{ , } z\in\mathbb{R} \mbox{ , }\vartheta\in\left[0,t\right]
\end{equation}
and its derivative with respect to $t$ are continuous in $0 \leq \vartheta \leq t$, then the FPT probability density function is the solution of the following integral equation: 
\begin{eqnarray}\label{integral1}
&&\mathbb{P}\left(X_1(t)\geq S\left|\mathbf{X}(0)=\mathbf{0}\right.\right)\\
&=&\int_0^t d\vartheta \; g\left(\vartheta\left|\right.\mathbf{0},0\right)
\mathbb{E}_{Z(\vartheta)}\left[\mathbb{P}\left(X_1(t)\geq S\left|X_1(\vartheta)=S,X_2(\vartheta)\right.\right)\right]\nonumber
\end{eqnarray}
where the distribution of $Z(\vartheta)$ is given by (\ref{densCond}).\\
The solution of (\ref{integral1}) exists and it is unique.
\end{thm}

\begin{proof}
Let us consider (\ref{FortetInt}) with $x_1=S$ and $x_2=-\infty$,
we get
\begin{eqnarray}\label{integral2}
&&\mathbb{P}\left(X_1(t)>S|\mathbf{X}(0)=\mathbf{0}\right)\\
&&= \int_{0}^t d\vartheta \int_{-\infty}^{+\infty} g_c\left((S,z),\vartheta\left|\right.\mathbf{0},0\right) 
\mathbb{P}\left(X_1(t)>S|X_1(\vartheta)=S,X_2(\vartheta)=z\right)dz.\nonumber
\end{eqnarray}

\noindent
Considering 
$$\mathbb{P}\left(X_2(T)<z,T<t \left|\mathbf{X}(0)=\mathbf{0}\right.\right)=
\int_0^t d\tau \; \mathbb{P}\left(X_2(\tau)<z\left|T=\tau,\mathbf{X}(0)=\mathbf{0}\right.\right)g\left(\tau\left|\mathbf{0},0\right.\right),$$ 
taking the derivatives with respect to $z$ and $t$, using (\ref{gc}), we get
\begin{equation} \label{gc2}
g_c\left((S,z),t\left|\mathbf{0},0\right.\right)=\frac{\partial}{\partial z}\mathbb{P}\left(X_2(T)<z\left|T=t;\mathbf{X}(0)=\mathbf{0}\right.\right)g\left(t\left|\mathbf{0},0\right.\right) \mbox{.}
\end{equation}

\noindent
Substituting (\ref{gc2}) in (\ref{integral2}) we get the integral equation (\ref{integral1}).
It is a first kind Volterra equation with regular kernel
$$k(t,\vartheta)=\mathbb{E}_{Z(\vartheta)}\left[\mathbb{P}\left(X_1(t)\geq S\left|X_1(\vartheta)=S,X_2(\vartheta)\right.\right)\right],$$
because $k(t,\vartheta)$ is bounded. In particular $k(t,t)$ does not vanish for any $t\geq 0$.\\
Due to the hypothesis (\ref{hypothesis}), the kernel of the Volterra equation (\ref{integral1})
and its derivative respect to $t$ are continuous for $0 \leq \vartheta \leq t$.\\
Similarly, the left hand side of equation (\ref{integral1}) and its derivative respect to $t$ are continuous for $t \geq 0$.
Furthermore $\mathbb{P}\left(X_1(0)\geq S\left|\mathbf{0},0\right.\right)=0$.\\
Thus, applying Theorem 5.1 of \cite{Linz}, we get the existence and uniqueness of the solution. 
\end{proof}

%\hspace{1mm}\\
\begin{cor}
The first passage time probability density of a Gauss-Markov process (\ref{sde}) satisfies the following equation
\begin{eqnarray} \label{Volterra}
&&1-\mbox{Erf}\left(\frac{S-m^{(1)}(t)}{\sqrt{2Q^{(11)}(t)}}\right)= \\ \nonumber
&&=\int_{0}^{t} d\vartheta g\left(\vartheta\left|\mathbf{0},0\right.\right)\mathbb{E}_{Z(\vartheta)}\left[1-\mbox{Erf}\left(\frac{S-m^{(1)}(t\left|\right.(S,X_2(\vartheta)),\vartheta)}{\sqrt{2Q^{(11)}(t\left|\right.(S,X_2(\vartheta)),\vartheta)}}\right)\right] \mbox{,} 
\end{eqnarray}
where $m^{(1)}(t) = m^{(1)}(t\left|\right.\mathbf{0},0)$ denotes the first component of the vector (\ref{vector}), $Q^{(11)}(t) = Q_t^{(11)}(t\left|\right.\mathbf{0},0)$ denotes the element on the upper left corner of the matrix (\ref{matrix}) and $\mbox{Erf}(x)$ denotes the error function \cite{Abr}.\\
The Volterra equation (\ref{Volterra}) admits a unique solution if
\begin{equation} \label{cond1}
\frac{\partial}{\partial t}\left(\frac{S-m^{(1)}(t\left|\right.(S,z),\vartheta)}{\sqrt{2Q^{(11)}(t\left|\right.(S,z),\vartheta)}}\right)
\end{equation}
is a continuous function of $t \geq \vartheta \geq 0$.
\end{cor}

\begin{proof}
Due to the Gaussianity of the process, we have
\begin{eqnarray*}
\mathbb{P}\left(X_1(t)\geq S\left|\mathbf{X}(t_0)=\mathbf{y}\right.\right)
&=&\int_{-\infty}^{+\infty} dx_2 \int_{S}^{+\infty} dx_1 f\left(\mathbf{x},t\left|\right.\mathbf{y},t_0\right)\\
&=&\frac{1}{2}\left(1-\mbox{Erf}\left(\frac{S-m^{(1)}(t\left|\right.\mathbf{y},t_0)}{\sqrt{2Q^{(11)}(t\left|\right.\mathbf{y},t_0)}}\right)\right)\mbox{.}
\end{eqnarray*}
Replacing this result into (\ref{integral1}), we obtain (\ref{Volterra}).\\

\noindent
Since (\ref{cond1}) is continuous for hypothesis, then
\begin{eqnarray*}
& \frac{\partial}{\partial t}\mathbb{P}\left(X_1(t)\geq S\left|\mathbf{X}(t_0)=\mathbf{y}\right.\right)
 =\frac{\partial}{\partial t}
    \frac{1}{2}\left(1-\mbox{Erf}\left(\frac{S-m^{(1)}(t\left|\right.\mathbf{y},t_0)}{\sqrt{2Q^{(11)}(t\left|\right.\mathbf{y},t_0)}}\right)\right)\\
& =-\frac{1}{\sqrt{\pi}}\exp\left\{-\left(\frac{S-m^{(1)}(t\left|\right.\mathbf{y},t_0)}
    {\sqrt{2Q^{(11)}(t\left|\right.\mathbf{y},t_0)}}\right)^2\right\}\frac{\partial}{\partial t}\left(\frac{S-m^{(1)}(t\left|\right.\mathbf{y},t_0)}
    {\sqrt{2Q^{(11)}(t\left|\right.\mathbf{y},t_0)}}\right)\mbox{.}
\end{eqnarray*} 
is continuous. Thus applying Theorem \ref{esistenza e unicit} we get the existence and uniqueness of the solution of (\ref{Volterra}).
\end{proof}

\begin{rmk}\label{C}
The term
$$\mathbb{P}\left(X_1(t)\geq S\left|X_1(\vartheta)=S,X_2(\vartheta)=z\right.\right)$$ 
represents the probability of being over the threshold $S$ after a time interval $(t-\vartheta)$, starting from the threshold itself. 
For a Gauss-Markov process it becomes
$$\mathbb{P}\left(X_1(t)\geq S\left|X_1(\vartheta)=S,X_2(\vartheta)=z\right.\right)= \frac{1}{2}\left[1-\mbox{Erf}\left(\frac{S-m^{(1)}(t\left|\right.(S,z),\vartheta)}{\sqrt{2Q^{(11)}(t\left|\right.(S,z),\vartheta)}}\right)\right] \mbox{.}$$
Thus under weak conditions of its well-definition, applying the l'Hopital's rule, the following limit
$$\lim_{\vartheta\rightarrow t}
\left\{1-\mbox{Erf}\left(\frac{S-m^{(1)}(t\left|\right.(S,z),\vartheta)}{\sqrt{2Q^{(11)}(t\left|\right.\vartheta,(S,z))}}\right)\right\}$$
assume a positive value $C\leq 2$. Then using the dominated convergence theorem we can conclude that 
\begin{equation}\label{limite}
\lim_{\vartheta \rightarrow t}
\mathbb{E}_{Z(\vartheta)}\left[1-\mbox{Erf}\left(\frac{S-m^{(1)}(t\left|\right.(S,X_2(\vartheta)),\vartheta)}{\sqrt{2Q^{(11)}(t\left|\right.(S,X_2(\vartheta)),\vartheta)}}\right)\right]=C .
\end{equation}
\end{rmk}

\begin{rmk}
Note that the random variable $X_2(\vartheta)$, that appear in the expectation (\ref{limite}), has values on an interval $[a,b]$ that changes depending on the features of the process (\ref{sde}).
\end{rmk}

\section{Gauss-Markov processes: a numerical algorithm} \label{num algo}

The complexity of equation (\ref{Volterra}) does not allow to get closed form solutions for $g$. Hence we pursue our study by introducing a numerical algorithm for its solution.\\

\noindent
Let us consider the partition $t_0=0<t_1<\ldots<t_N=t$ of the time interval $[0,t]$ with step $h=t_k-t_{k-1}$ for $k=1,\ldots,N$.\\
Discretizing integral equation (\ref{Volterra}) via Euler method, we have:
\begin{eqnarray}\label{Eulero}
&&1-\mbox{Erf}\left(\frac{S-m^{(1)}(t_k)}{\sqrt{2Q^{(11)}(t_k)}}\right)\\ \nonumber
&=&\sum_{j=1}^k \hat{g}\left(t_j\left|\mathbf{0},0\right.\right)
\mathbb{E}_{Z(t_j)}\left[1-\mbox{Erf}\left(\frac{S-m^{(1)}(t_k\left|\right.(S,X_2(t_j)),t_j)}{\sqrt{2Q^{(11)}(t_k\left|\right.(S,X_2(t_j)),t_j)}}\right)\right] h
\end{eqnarray}
for $k=1,\ldots,N$.\\

\noindent
Equation (\ref{Eulero}) gives the following algorithm for the numerical approximation $\hat{g}\left(\tau\left|\mathbf{0},0\right.\right)$ of $g\left(\tau\left|\mathbf{0},0\right.\right)$, $\tau\in (0,t)$.\\

\noindent
\textbf{Step $1$}
$$\hat{g}\left(t_1\left|\right.\mathbf{0},0\right)= \frac{1}{Ch}\left[1-\mbox{Erf}\left(\frac{S-m^{(1)}(t_1)}{\sqrt{2Q^{(11)}(t_1)}}\right)\right],$$
where $C$ is given by (\ref{limite}).

\hspace{1mm}\\
\textbf{Step $k$}, $k>1$
\begin{eqnarray}\label{algoritmo}
\hat{g}\left(t_k\left|\right.\mathbf{0},0\right)
&=&\frac{1}{Ch}\left\{1-\mbox{Erf}\left(\frac{S-m^{(1)}(t_k)}{\sqrt{2Q^{(11)}(t_k)}}\right)\right\} \\ \nonumber
&-& \frac{1}{C} \sum_{j=1}^{k-1}\hat{g}\left(t_j\left|\right.\mathbf{0},0\right) \mathbb{E}_{Z(t_j)}\left[1-\mbox{Erf}\left(\frac{S-m^{(1)}(t_k\left|\right.(S,X_2(t_j)),t_j)}{\sqrt{2Q^{(11)}(t_k\left|\right.(S,X_2(t_j)),t_j)}}\right)\right] \mbox{.}
\end{eqnarray}
Note that the first term on the r.h.s. is obtained for $j=k$.

\hspace{1mm}\\
To sum up, the FPT probability density function in the knots $t_0,t_1,\ldots, t_N$ is the solution of a linear system $\mathbf{L}\mathbf{\hat{g}}=\mathbf{b}$ where

$$\mathbf{b}=\left(\begin{array}{c}
             1-\mbox{Erf}\left(\frac{S-m^{(1)}(t_1)}{\sqrt{2Q^{(11)}(t_1)}}\right)\\
             \vdots\\     
             1-\mbox{Erf}\left(\frac{S-m^{(1)}(t_N)}{\sqrt{2Q^{(11)}(t_N)}}\right)
             \end{array}\right)\mbox{, }\hspace{1.5mm}
  \mathbf{\hat{g}}=\left(\begin{array}{c}
             \hat{g}\left(t_1\left|\right.\mathbf{0},0\right)\\
             \vdots\\
             \hat{g}\left(t_N\left|\right.\mathbf{0},0\right)
             \end{array}\right)$$
and
$$\mathbf{L}=\left(\begin{array}{ccccc}
    C h \\
    \theta_{2,1} h & C h \\
    \theta_{3,1} h & \theta_{3,2} h & C h \\
    \vdots & \vdots & & \ddots\\
    \theta_{N,1} h & \theta_{N,2} h & \cdots & \cdots & C h
    \end{array}\right),$$

\hspace{1mm}\\

\noindent with
\begin{equation}\label{attesa} \theta_{k,j}=\mathbb{E}_{Z(t_j)}\left[1-\mbox{Erf}\left(\frac{S-m^{(1)}(t_k\left|(S,X_2(t_j)),t_j\right.)}{\sqrt{2Q^{(11)}(t_k\left|(S,X_2(t_j)),t_j\right.)}}\right)\right]
\end{equation} 
for $k=1,\ldots,N$ and $j=1,\ldots,k$.\\
%\hspace{1mm}\\    

To evaluate the expected value (\ref{attesa})
for $k=1,\ldots,N$ and $j=1,\ldots,k$, we make use of the following Monte Carlo method.\\
We repeatedly simulate the bivariate process until the first component crosses the boundary and we collect the sequence $\left\{Z_i,i=1,\ldots M\right\}$ of i.i.d random variables with probability distribution function (\ref{densCond}) with $t=t_j$. Then we compute the sample mean 
\begin{equation}\label{attesaMC}
\hat{\theta}_{k,j}=
1-\frac{\sum_{i=1}^{M}\mbox{Erf}\left(\frac{S-m^{(1)}(t_k\left|(S,Z_i),t_j\right.)}{\sqrt{2Q^{(11)}(t_k\left|(S,Z_i),t_j\right.)}}\right)}{M} \mbox{.}
\end{equation} 
Here $M$ is the sample size.

%\hspace{1mm}\\
The following theorem proves that this algorithm converges.
In order to simplify the notations of the theorem, let us first define 
$$\psi(z;t_k,t_j):=\mbox{Erf}\left(\frac{S-m^{(1)}(t_k\left|\right.(S,z),t_j)}{\sqrt{2Q^{(11)}(t_k\left|\right.(S,z),t_j)}}\right)$$
for $k=1,\ldots,N$, $j=1,\dots,k$.
Then it holds

\begin{thm}\label{error}
If the sample size $M$ for the Monte Carlo method is such that the error $|\lambda|= h^2$ at a confidence level $\alpha$ and if there exists a constant $a$, such that for all $h > 0$
\begin{equation}\label{hp}
\max_{1\leq k\leq N,1\leq j\leq k-1} \mathbb{E}_{Z(t_j)}\left[\psi(X_2(t_j),t_k,t_j)-\psi(X_2(t_j),t_{k-1},t_j)\right]\leq a h,
\end{equation}
then the error $\epsilon_k=\hat{g}\left(t_k\left|\right.\mathbf{0},0\right)-g\left(t_k\left|\right.\mathbf{0},0\right)$ of the proposed algorithm at the discretization knots $t_k$, for $k=1,2,\ldots$ is
$\left|\epsilon_k\right|=O(h)$
at the same confidence level $\alpha$.
\end{thm}

\begin{proof}
The Euler method and the Monte Carlo method applied to (\ref{Volterra}) give
\begin{eqnarray}\label{stima}
1-\mbox{Erf}\left(\frac{S-m^{(1)}(t_k)}{\sqrt{2Q^{(11)}(t_k)}}\right)=
\sum_{j=1}^{k} h \hat{g}\left(t_j\left|\mathbf{0},0\right.\right)\hat{\theta}_{k,j}
\end{eqnarray}
while (\ref{Volterra}) can be rewritten as
\begin{eqnarray}\label{esatta}
1-\mbox{Erf}\left(\frac{S-m^{(1)}(t_k)}{\sqrt{2Q^{(11)}(t_k)}}\right)=
\sum_{j=1}^{k} h g\left(t_j\left|\mathbf{0},0\right.\right)\left(\hat{\theta}_{k,j}+\lambda\right)+ \delta(h,t_k)
\end{eqnarray}
where $\delta(h,t_k)$ denotes the error of Euler method and $\lambda$ indicates the error of the Monte Carlo method at confidence level $\alpha$.\\
Subtracting (\ref{esatta}) from (\ref{stima}) we obtain
\begin{equation} \label{errore}
\delta(h,t_k)=h\sum_{j=1}^{k}\left(\hat{\theta}_{k,j}\epsilon_j+\lambda g\left(t_j\left|\mathbf{0},0\right.\right)\right).
\end{equation}
Differencing (\ref{errore}) and using (\ref{limite}), we get
$$\delta(h,t_k)-\delta(h,t_{k-1})=h\sum_{j=1}^{k-1}\left(\hat{\theta}_{k,j}-\hat{\theta}_{k-1,j} \right)\epsilon_j+hC\epsilon_k+g\left(t_k\left|\mathbf{0},0\right.\right)\lambda$$
or equally
$$\epsilon_k=\frac{1}{Ch}\left[\delta(h,t_k)-\delta(h,t_{k-1})\right]-\frac{1}{C}\sum_{j=1}^{k-1}\left(\hat{\theta}_{k,j}-\hat{\theta}_{k-1,j} \right)\epsilon_j -\frac{g\left(t_k\left|\mathbf{0},0\right.\right)\lambda}{hC}.$$
Then, due to the hypothesis (\ref{hp}) and to the large number law, choosing $M$ large enough, we have
$$\left|\epsilon_k\right| \leq \frac{1}{Ch}\left|\delta(h,t_k)-\delta(h,t_{k-1})\right|+ \frac{a h}{C}\sum_{j=1}^{k-1}\epsilon_j +\frac{g\left(t_k\left|\mathbf{0},0\right.\right)|\lambda|}{hC}.$$
Finally, observing that the error of Euler method is $\left|\delta(h,t)\right| = O(h^2)$, choosing $M$ such that the error of the Monte Carlo method is $|\lambda |= h^2$ and applying Theorem 7.1 of \cite{Linz}, we get the thesis.
\end{proof}

\begin{rmk}
In the autonomous case, hypothesis (\ref{hp}) is verified.\\
Indeed
\begin{eqnarray}\label{inequality1}
\psi\left(z;t_k,t_j\right)-\psi\left(z;t_{k-1},t_j\right)
&=&\frac{2}{\sqrt{\pi}}\int^{\beta\left((S,z),t_k,t_j\right)}_{\beta\left((S,z),t_{k-1},t_j\right)} e^{-y^2/2} dy\\ \nonumber
&\leq&a_1\left[\beta\left((S,z),t_k,t_j\right)-\beta\left((S,z),t_{k-1},t_j\right)\right]
\end{eqnarray}
where
$$\beta\left((S,z),t_i,t_l\right)=\frac{S-m^{(1)}(t_i\left|\right.(S,z),t_l)}{\sqrt{2Q^{(11)}(t_i\left|\right.(S,z),t_l)}}$$
Using equation (\ref{newmatrix}), we get that
$$Q^{(11)}(t_k\left|\right.(S,z),t_j)=Q^{(11)}(t_{k-1}\left|\right.(S,z),t_j)+\int_0^h e^{Au}\mathbf{M}\mathbf{M}^{'}e^{A^{'}u}du .$$
Hence inequality (\ref{inequality1}) becomes
\begin{eqnarray}\label{inequality2}
\psi\left(z;t_k,t_j\right)-\psi\left(z;t_{k-1},t_j\right)\leq a_2\frac{m^{(1)}(t_k\left|\right.(S,z),t_j)-m^{(1)}(t_{k-1}\left|\right.(S,z),t_j)}{\sqrt{Q^{(11)}(t_{k-1}\left|\right.(S,z),t_j)}},
\end{eqnarray}
where $a_2$ is a new constant.\\
Using equation (\ref{newvector}), $\mathbf{y}=(S,z)$ and $t_k-t_{k-1}=h$, $k=1,\ldots N$, we can conclude that
\begin{eqnarray*}
&&\mathbf{m}\left(t_k|\mathbf{y},t_j\right)-\mathbf{m}\left(t_{k-1}|\mathbf{y},t_j\right)\\
&\leq& a_3\left\{\left(A(t_k-t_{k-1})\right)\mathbf{y}+\int_{t_j}^{t_k}\left(I+A(t_k-u)\right)\mathbf{M}du-\int_{t_j}^{t_{k-1}}\left(I+A(t_{k-1}-u)\right)\mathbf{M}du\right\}\\
&=&
a_3\left\{(A\mathbf{y}+\mathbf{M})h-A\left(k-j-\frac{1}{2}\right)\mathbf{M}h^2\right\}.
\end{eqnarray*}

Therefore inequality (\ref{inequality2}) becomes
\begin{eqnarray*}
\psi\left(z,t_k,t_j\right)-\psi\left(z,t_{k-1},t_j\right)
\leq\frac{a_4h+a_5h^2}{\sqrt{Q_{t_{k-1}}^{(11)}((S,z),t_j)}}
=O(h).
\end{eqnarray*}

\end{rmk}

\section{Examples}
In this section we apply the algorithm presented in Sections \ref{num algo} to some examples. Firstly we consider an IBM and an IOU Process. Recent examples of application of the IBM and the IOU processes have appeared in the metrological literature \cite{panfiloTavellaZucca} to model the error of atomic clock. Then we will consider the two-compartment model of a neuron \cite{Lansky}, whose underlying process is a bivariate Ornstein Uhlenbeck process.

\subsection{Integrated Brownian Motion}

The Integrated Brownian Motion by itself is not a Gauss-Markov process because it is not a Markov process. However we can study this one dimensional process as a bivariate process together with a Standard Brownian motion, as follows
\begin{equation}\label{IBM}
\left\{\begin{array}{lll}
        dX_1(t)&=&X_2(t)dt\\
        \hspace{5mm}\\
        dX_2(t)&=&dB_t,
        \end{array}\right.
\end{equation}
with $\mathbf{X}(0)=\mathbf{0}$.\\
The process (\ref{IBM}) is a particular case of the Gauss-Markov process (\ref{sde}), where 
$$\mathbf{A}(t)=\left(\begin{array}{cc} 0 & 1 \\ 0 & 0 \end{array}\right) \mbox{, } \mathbf{M}(t)=\left(\begin{array}{ll} 0 \\ 0 \end{array}\right) \mbox{ and }  
\mathbf{G}(t)=\left(\begin{array}{cc} 0 & 0 \\ 0 & 1 \end{array}\right) \mbox{.}$$

There exist analytical solutions of the first passage time problem of the integrated component of the process (\ref{IBM}), but they are not efficient because they involve multiple integrals \cite{Goldman} or suppose particular symmetry properties \cite{giornoSymmetrytwoDim}.\\
Hence, we numerically solve the first passage time problem for an IBM using the algorithm proposed in Section \ref{num algo}. A first attempt in this direction was discussed in \cite{Lidia}.
In this instance the FPT probability density function through a boundary S in the knots $t_0,t_1,\ldots t_N$ is solution of a linear system $\mathbf{L}\mathbf{g}=\mathbf{b}$ where

$$\mathbf{b}=\left(\begin{array}{c}
             1-\mbox{Erf}\left(\frac{\sqrt{6}S}{2t_1^{3/2}}\right)\\
             \vdots\\     
             1-\mbox{Erf}\left(\frac{\sqrt{6}S}{2t_N^{3/2}}\right)
             \end{array}\right) \mbox{, }\hspace{1.5mm}
  \mathbf{g}=\left(\begin{array}{c}
             g\left(t_1\left|\right.\mathbf{0},0\right)\\
             \vdots\\
             g\left(t_N\left|\right.\mathbf{0},0\right)
             \end{array}\right)$$
and
$$\mathbf{L}=\left(\begin{array}{ccccc}
    2 h \\
    \theta_{2,1} h & 2 h \\
    \theta_{3,1} h & \theta_{3,2} h & 2 h \\
    \vdots & \vdots & & \ddots\\
    \theta_{N,1} h & \theta_{N,2} h & \cdots & \cdots & 2 h
    \end{array}\right),$$

\hspace{1mm}\\

\noindent with $\theta_{k,j}=\mathbb{E}_{Z(t_j)}\left[1+\mbox{Erf}\left(\frac{\sqrt{6}X_2(t_j)}{2\sqrt{(t_k-t_j)h}}\right)\right]$ 
for $k,j=1,\ldots,N$.\\

\noindent
Note that in this case the constant $C$ defined in Remark \ref{C} is equal to 2 and the range of the random variable $X_2(T)$  is $[0,\infty]$.\\

\noindent
In Figure \ref{IBMfig} we show the FPT probability density function of an IBM through a boundary $S$ for three different values of the boundary.
\begin{figure}[!h]
%	\begin{center}
\includegraphics[width=0.90\textwidth]{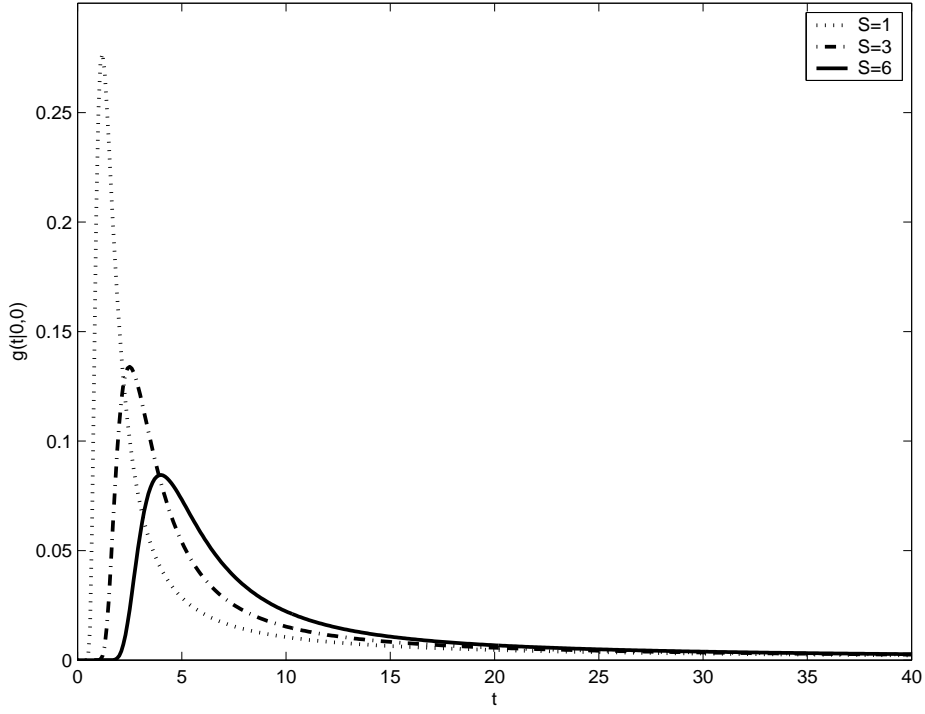}
%	\end{center}
	\caption{Evaluation of the FPT probability density function for an IBM through three different boundaries: $S=1$ (dotted), $S=3$ (dashdot), $S=6$ (solid).}
	\label{IBMfig}
\end{figure}

\subsection{Integrated Ornstein Uhlenbeck Process}

As the IBM, the IOU Process is not a Markov process and it should be studied as a bivariate process together with an Ornstein Uhlenbeck Process, as follows
\begin{equation}\label{IOU}
\left\{\begin{array}{lll}
        dX_1(t)&=&X_2(t)dt\\
        \hspace{5mm}\\
        dX_2(t)&=&\left(-\alpha X_2(t) + \mu \right)dt + \sigma dB_t,
        \end{array}\right. 
\end{equation}
with $\mathbf{X}(0)=\mathbf{0}$.
The process (\ref{IOU}) is a Gauss-Markov process (\ref{sde}), where 
$$\mathbf{A}(t)=\left(\begin{array}{cc} 0 & 1 \\ 0 & -\alpha \end{array}\right) \mbox{, } \mathbf{M}(t)=\left(\begin{array}{ll} 0 \\ \mu \end{array}\right) \mbox{ and }  
\mathbf{G}(t)=\left(\begin{array}{cc} 0 & 0 \\ 0 & \sigma \end{array}\right) \mbox{.}$$

\noindent
Note that in this case the constant $C$ defined in Remark \ref{C} is equal to 1  and the range of the random variable $X_2(T)$ is $[-\infty,\infty]$.\\

\noindent
In Figure \ref{IOUfig1} we show the FPT probability density function of the IOU through a boundary $S=6$, for $\mu=0.01$, $\sigma=1$ and three different values of the parameter $\alpha$.

\begin{figure}[!h]
%	\begin{center}
		\includegraphics[width=0.90\textwidth]{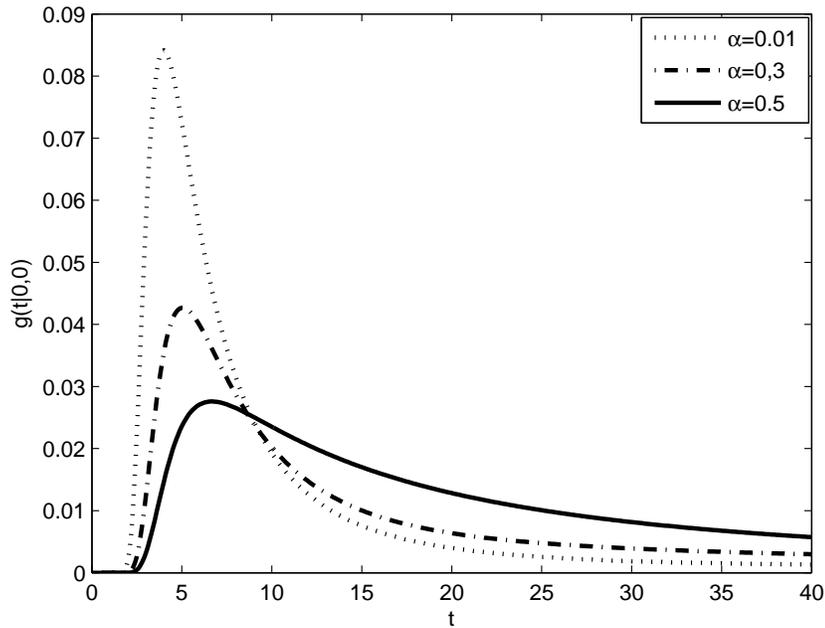}
%	\end{center}
	\caption{Evaluation of the FPT probability density function of the IOU through a boundary $S=6$ for $\mu=0.01$, $\sigma=1$ and three different values of the parameter $\alpha$: $\alpha=0.01$ (dotted), $\alpha=0.3$ (dashdot), and $\alpha=0.5$  (solid).}
	\label{IOUfig1}
\end{figure}

\noindent Note that the curve for $\alpha=0.01$ in Figure \ref{IOUfig1} is very similar to the curve for $S=6$ in Figure \ref{IBMfig}, indeed if $\mu\rightarrow 0$ and $\alpha\rightarrow 0$ IOU converges to a standard IBM.

\subsection{Two-compartment model}

One dimensional neuronal models \cite{RiccLectureNotes} identify the membrane potential values on the different parts of the neuron with those assumed in the trigger zone. To improve the model \cite{Lansky} proposes a two dimensional approach. The membrane potential of a neuron is described by means of a bivariate stochastic process, whose components represent the depolarizations of two distinct parts, the trigger zone and the dentritic one. \\

\noindent
Let $(X_1(t))_{t\geq 0}$ and $(X_2(t))_{t\geq 0}$ be the stochastic processes associated to the depolarization of the trigger zone and the dendritic one, respectively. Then, assuming that external inputs, of intensity $\mu$ and variability $\sigma$, influence only the second compartment and taking the interconnection between the parts into account, we obtain the following stochastic model
\begin{equation}\label{model}
\left\{\begin{array}{lll} dX_1(t)=\left\{-\alpha X_1(t) + \alpha_r \left[X_2(t)-X_1(t)\right]\right\}dt\\
                            \vspace{0.1mm}\\
                            dX_2(t)=\left\{-\alpha X_2(t) + \alpha_r \left[X_1(t)-X_2(t)\right] + \mu\right\}dt + \sigma dB_t \end{array} \right.
\end{equation}
with $\mathbf{X}(0)=\mathbf{0}$ and where $\alpha$ and $\alpha_r$ are constant related to the spontaneous membrane potential decay and to the intensity of the connection between the two compartments, respectively (cf. Fig \ref{BiModel}). 

\begin{figure}[!h]
%	\begin{center}
		\includegraphics[width=0.90\textwidth]{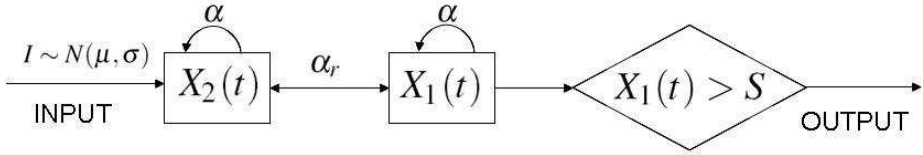}
%	\end{center}
	\caption{Schematic representation of the two-compartment approach}
	\label{BiModel}
\end{figure}

\noindent
The process (\ref{model}) is an Ornstein Uhlenbeck two-dimensional process, particular case of the Gauss-Markov process (\ref{sde}) where 
$$\mathbf{A}(t)=\left(\begin{array}{cc} -\alpha-\alpha_r & \alpha_r \\ \alpha_r & -\alpha-\alpha_r \end{array}\right) \mbox{, } \mathbf{M}(t)=\left(\begin{array}{ll} 0 \\ \mu \end{array}\right) \mbox{ and }  
\mathbf{G}(t)=\left(\begin{array}{cc} 0 & 0 \\ 0 & \sigma \end{array}\right) \mbox{.}$$
Note that in this case the constant $C$ defined in Remark \ref{C} is equal to 2  and the range of the random variable $X_2(T)$ is $[kS,\infty]$, where 
	\[k=\frac{\alpha+\alpha_r}{\alpha_r}.
\]

\noindent
Assuming that after each spike the system is reset to its initial value, the time between two spikes, i.e. the time when the membrane potential changes its dynamics with a sudden hyperpolarization, is described by the FPT of the first component through a boundary $S$.\\
In \cite{Lansky} this model was studied using simulation techniques, but applying the algorithm proposed in Section \ref{num algo} we can compute the interspike intervals distribution as the FPT probability density function of the first component of $\mathbf{X}(t)$.
In Figure \ref{BICOMPfig1} we show the FPT density for a set of values of the input.\\

\begin{figure}[!h]
%	\begin{center}
		\includegraphics[width=0.90\textwidth]{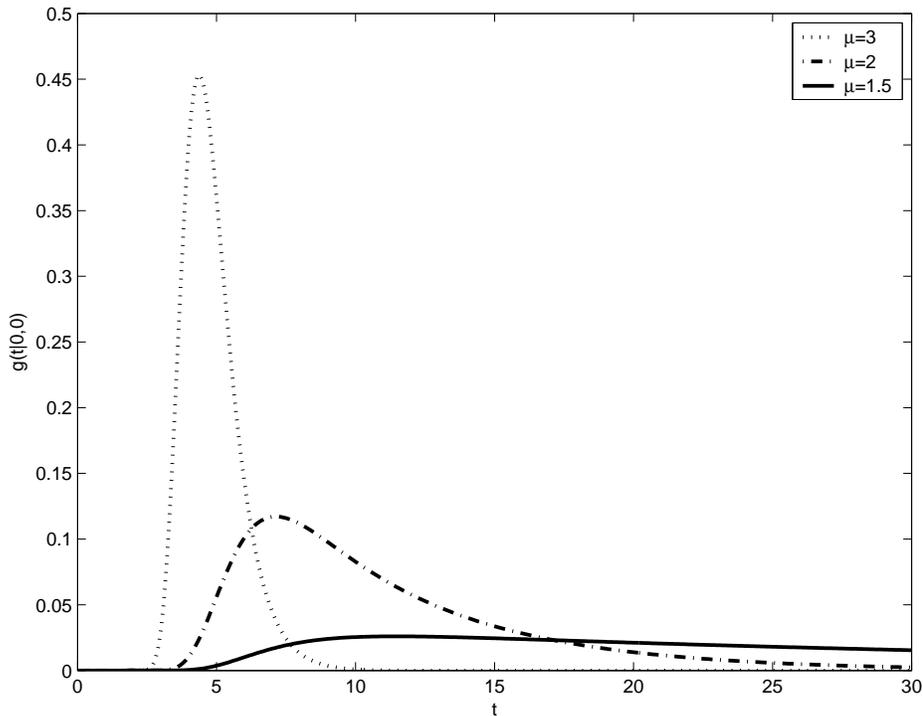}
%	\end{center}
	\caption{Evaluation of the FPT probability density function through a boundary $S=6$ for a two-compartment stochastic model of a neuron, choosing $\alpha=0.33$, $\alpha_r=0.2$, $\sigma=1$ and three different values of $\mu$: $\mu=1.5$ (solid), $\mu=2$ (dashdot), and $\mu=3$ (dotted).}
	\label{BICOMPfig1}
\end{figure}

\section{Numerical algorithm versus a totally simulation algorithm}

The introduced numerical method involves a Monte Carlo estimation to evaluate the expected value (\ref{attesa}). One may wonder about the advantages of the proposed method with respect to a totally simulation algorithm. Indeed it is easy to simulate $M$ sample paths of the considered bivariate process to get a sample of FPTs. These FPTs could be used to draw histograms or their continuous approximations. However this approach is computationally expensive. Indeed it requires large samples to give reliable results. Moreover the estimation of the tails of the distribution is scarcely reliable and time consuming. 

On the contrary the numerical method proposed here requests weak computational efforts despite the presence of the Monte Carlo method and is applied to estimate a specific integral, i.e. the expectation (\ref{attesa}). Indeed, also a small number of trajectories guarantees reliable results.

In Figure \ref{IBMhist} we compare the accuracy of the results obtained with the two methods. We simulate $M$ sample paths of the IBM in order to determine a sample of $M$ FPTs and we use it to draw the corresponding histogram. The same sample is used to compute the sample mean (\ref{attesaMC}) to get the FPT probability density function via the numerical method. The choice $M=1000$ gives reliable results in both cases. However, when $M=50$ the histogram is rude while the numerical method does not loose its reliability. 
We further underline how the computational time to build the histogram or to draw the FPT density with the proposed numerical method are comparable, for the same value of $M$. 

In Figure \ref{IBM1000_50N} we show the shapes of the FPT probability density function obtained via the proposed numerical algorithm. We use a sample of size $M=50$ (solid line) and $M=1000$ (dash line) to compute (\ref{attesaMC}). Their differences are negligible.
 
\begin{figure}
%  \centering
\subfloat[$M=50$ ]{\label{IBM50}
  \includegraphics[width=0.9\textwidth, height=8cm]{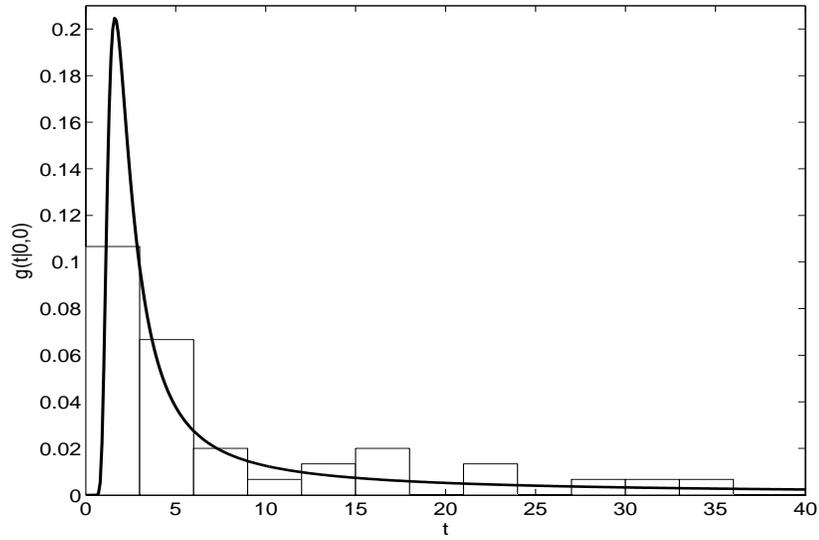}} \\               
  \subfloat[$M=1000$ ]{\label{IBM1000}
  \includegraphics[width=0.9\textwidth, height=8cm]{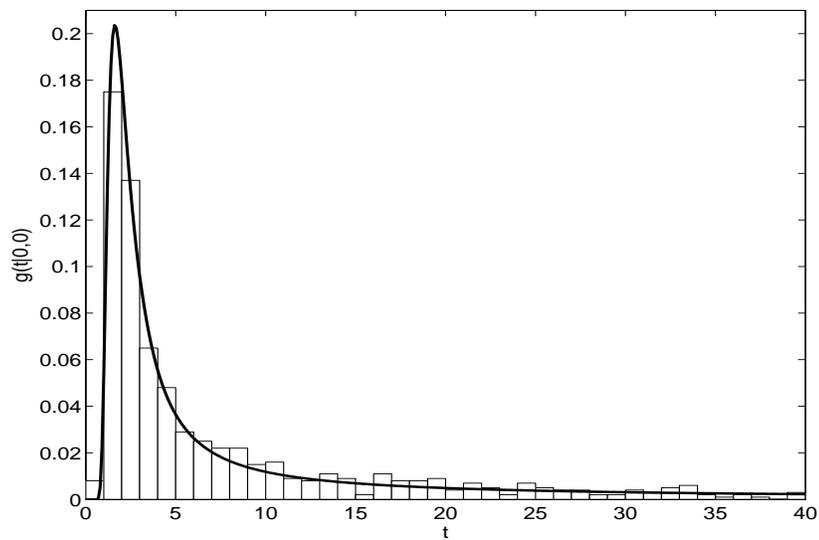}}
  \caption{FPT probability density function for the IBM obtained via the proposed numerical method and the corresponding histogram. Sample of size $M$ has been used to build the histogram and to evaluate (\ref{attesaMC}) in the numerical method: (a) $M=50$ (b) $M=1000$.}
  \label{IBMhist}
\end{figure}

\begin{figure}[!h]
%	\begin{center}
		\includegraphics[width=0.90\textwidth]{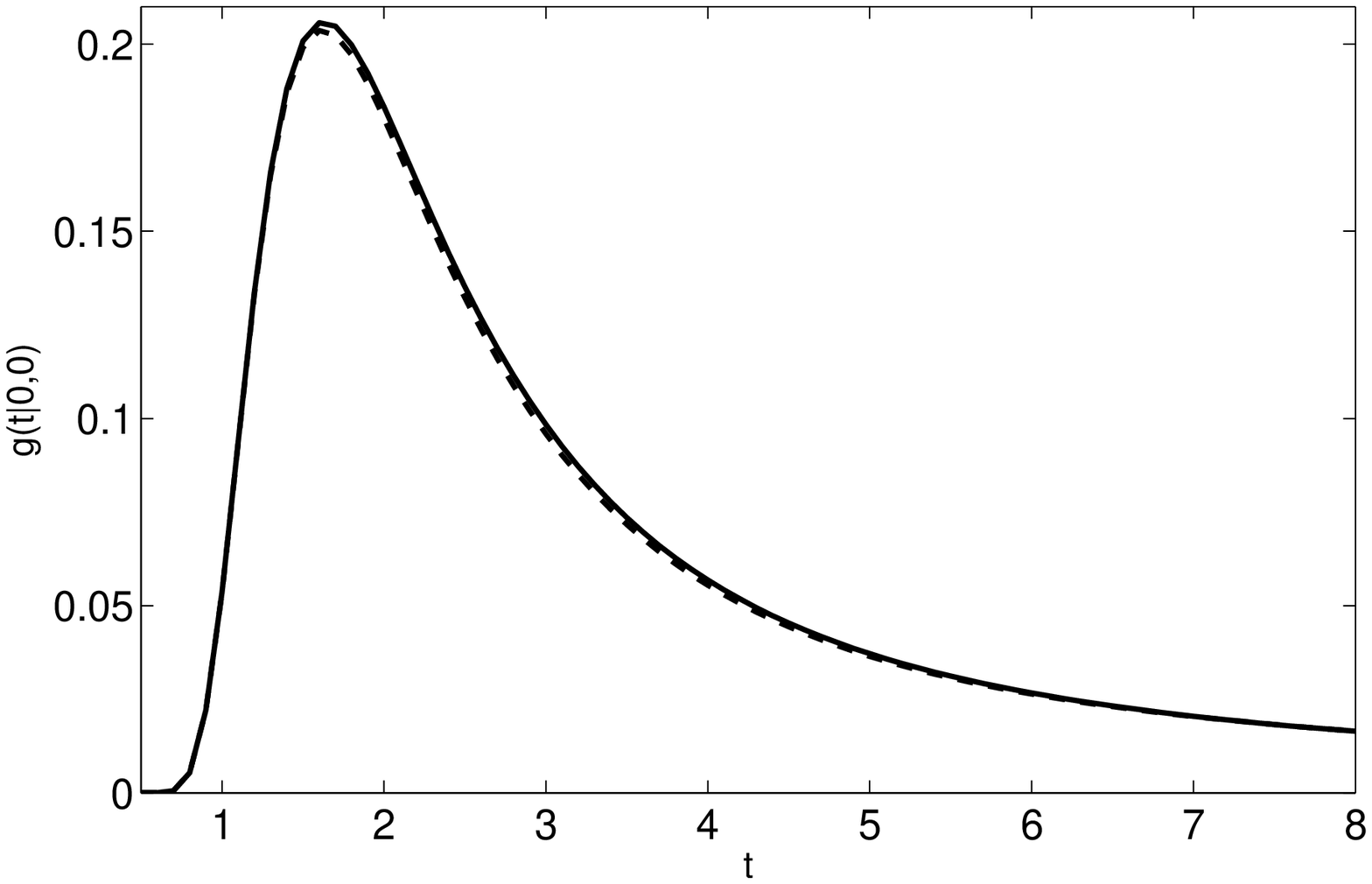}
%	\end{center}
	\caption{FPT probability density function for the IBM obtained via the proposed numerical method computed using a sample of size $M=50$ (solid line) and $M=1000$ (dash line) for the sample mean (\ref{attesaMC}).}
	\label{IBM1000_50N}
\end{figure}

\section{Conclusion}
We studied the FPT of one component of a bivariate diffusion process through a boundary. We wrote a new integral equation for the FPT probability density function proving its existence and uniqueness. A numerical algorithm for its solution was developed proving its convergence properties. The algorithm was applied to a set of processes of interest for various applications. Advantages of the method with respect to a totally simulation algorithm are discussed.

The crossing problem for one component of a multivariate process or the case of random initial value can be treated as extensions of the proposed equations and of the numerical algorithm.

\section{Acknowledgements}
Work supported in part by MIUR Project PRIN-Cofin 2008.



%% References


%\bibliographystyle{model1a-num-names}



\end{document}